\documentclass[11pt,onecolumn,draftclsnofoot,peerreview]{IEEEtran}
\ifCLASSINFOpdf
\else
\fi
\usepackage[cmex10]{amsmath}
\usepackage{amssymb,amsthm}

\newcommand{\pp}{\mathbb{P}}

\newcommand{\ee}{\mathbb{E}\,}
\newcommand{\ii}{\mathcal{I}}

\newcommand{\rr}{\mathbb{R}\,}

\newcommand{\eps}{\varepsilon}

\newcommand{\ch}{\mathcal{H}}
\newcommand{\one}{\mathbf{1}}

\newcommand{\bby}{\mbox{{\boldmath $\mathcal{Y}$}}}
\newcommand{\bbw}{\mbox{{\boldmath $\mathcal{W}$}}}

\newcommand{\by}{\mathbf{Y}}
\newcommand{\bw}{\mathbf{W}}

\newcommand{\cf}{\mathcal{F}}

\renewenvironment{proof}[1][\proofname]{\par \normalfont \trivlist
\item[\hskip\labelsep\itshape #1]\ignorespaces
}{%
\hspace*{\fill}$\Box$ \endtrivlist }
\renewcommand{\proofname}{{\bf Proof}}

\makeatletter
\def\newrmtheorem#1{\@ifnextchar[{\@rmothm{#1}}{\@rmnthm{#1}}}

\def\@rmnthm#1#2{%
\@ifnextchar[{\@rmxnthm{#1}{#2}}{\@rmynthm{#1}{#2}}}

\def\@rmxnthm#1#2[#3]{\expandafter\@ifdefinable\csname #1\endcsname
{\@definecounter{#1}\@addtoreset{#1}{#3}%
\expandafter\xdef\csname the#1\endcsname{\expandafter\noexpand
  \csname the#3\endcsname \@rmthmcountersep \@rmthmcounter{#1}}%
\global\@namedef{#1}{\@rmthm{#1}{#2}}\global\@namedef{end#1}{\@endrmtheorem}}}

\def\@rmynthm#1#2{\expandafter\@ifdefinable\csname #1\endcsname
{\@definecounter{#1}%
\expandafter\xdef\csname the#1\endcsname{\@rmthmcounter{#1}}%
\global\@namedef{#1}{\@rmthm{#1}{#2}}\global\@namedef{end#1}{\@endrmtheorem}}}

\def\@rmothm#1[#2]#3{\expandafter\@ifdefinable\csname #1\endcsname
  {\global\@namedef{the#1}{\@nameuse{the#2}}%
\global\@namedef{#1}{\@rmthm{#2}{#3}}%
\global\@namedef{end#1}{\@endrmtheorem}}}

\def\@rmthm#1#2{\refstepcounter
    {#1}\@ifnextchar[{\@rmythm{#1}{#2}}{\@rmxthm{#1}{#2}}}

\def\@rmxthm#1#2{\@beginrmtheorem{#2}{\csname the#1\endcsname}\ignorespaces}
\def\@rmythm#1#2[#3]{\@opargbeginrmtheorem{#2}{\csname
       the#1\endcsname}{#3}\ignorespaces}

\def\@rmthmcounter#1{\noexpand\arabic{#1}}
\def\@rmthmcountersep{}
\def\@beginrmtheorem#1#2{\rm \trivlist
      \item[\hskip \labelsep{\bf #1\ #2\thmrmcounterend}]}
\def\@opargbeginrmtheorem#1#2#3{\rm \trivlist
      \item[\hskip \labelsep{\bf #1\ #2\ (#3)\thmrmcounterend}]}
\def\@endrmtheorem{\endtrivlist}
\def\thmrmcounterend{\hskip 0em\relax}
\def\newrmwntheorem#1#2{\expandafter\@ifdefinable\csname #1\endcsname%
\global\@namedef{#1}{\@rmwnthm{#1}{#2}}%
\global\@namedef{end#1}{\@endrmwntheorem}}

\def\newsltheorem#1{\@ifnextchar[{\@slothm{#1}}{\@slnthm{#1}}}

\def\@slnthm#1#2{%
\@ifnextchar[{\@slxnthm{#1}{#2}}{\@slynthm{#1}{#2}}}

\def\@slxnthm#1#2[#3]{\expandafter\@ifdefinable\csname #1\endcsname
{\@definecounter{#1}\@addtoreset{#1}{#3}%
\expandafter\xdef\csname the#1\endcsname{\expandafter\noexpand
  \csname the#3\endcsname \@slthmcountersep \@slthmcounter{#1}}%
\global\@namedef{#1}{\@slthm{#1}{#2}}\global\@namedef{end#1}{\@endsltheorem}}}

\def\@slynthm#1#2{\expandafter\@ifdefinable\csname #1\endcsname
{\@definecounter{#1}%
\expandafter\xdef\csname the#1\endcsname{\@slthmcounter{#1}}%
\global\@namedef{#1}{\@slthm{#1}{#2}}\global\@namedef{end#1}{\@endsltheorem}}}

\def\@slothm#1[#2]#3{\expandafter\@ifdefinable\csname #1\endcsname
  {\global\@namedef{the#1}{\@nameuse{the#2}}%
\global\@namedef{#1}{\@slthm{#2}{#3}}%
\global\@namedef{end#1}{\@endsltheorem}}}

\def\@slthm#1#2{\refstepcounter
    {#1}\@ifnextchar[{\@slythm{#1}{#2}}{\@slxthm{#1}{#2}}}

\def\@slxthm#1#2{\@beginsltheorem{#2}{\csname the#1\endcsname}\ignorespaces}
\def\@slythm#1#2[#3]{\@opargbeginsltheorem{#2}{\csname
       the#1\endcsname}{#3}\ignorespaces}

\def\@slthmcounter#1{.\noexpand\arabic{#1}}
\def\@slthmcountersep{}
\def\@beginsltheorem#1#2{\sl \trivlist
      \item[\hskip \labelsep{\bf #1\ #2\thmslcounterend}]}
\def\@opargbeginsltheorem#1#2#3{\sl \trivlist
      \item[\hskip \labelsep{\bf #1\ #2\ (#3)\thmslcounterend}]}
\def\@endsltheorem{\endtrivlist}
\def\thmslcounterend{\hskip 0em\relax}
\def\newslwntheorem#1#2{\expandafter\@ifdefinable\csname #1\endcsname%
\global\@namedef{#1}{\@slwnthm{#1}{#2}}%
\global\@namedef{end#1}{\@endslwntheorem}}

\makeatother

\newsltheorem{theorem}{Theorem}[section]
\newsltheorem{proposition}[theorem]{Proposition}
\newsltheorem{lemma}[theorem]{Lemma}
\newsltheorem{corollary}[theorem]{Corollary}

\newrmtheorem{definition}[theorem]{Definition}
\newrmtheorem{example}[theorem]{Example}
\newrmtheorem{remark}[theorem]{Remark}
\newrmtheorem{problem}[theorem]{Problem}
\newrmtheorem{algorithm}[theorem]{Algorithm}
\newrmtheorem{cond}[theorem]{Condition}
\newrmtheorem{convention}[theorem]{Convention}

\begin{document}
%
%
\title{Approximation of Nonnegative Systems by Finite Impulse Response Convolutions}%
%

\author{Lorenzo Finesso,
        Peter Spreij
\thanks{L. Finesso is with the Institute of Biomedical Engineering,
National Research Council, CNR-ISIB, Padova, email: (lorenzo.finesso@isib.cnr.it)}%
\thanks{P. Spreij is with the Korteweg-de Vries Institute for Mathematics,
Universiteit van Amsterdam, Amsterdam, The Netherlands e-mail: (spreij@uva.nl)}
\thanks{Corresponding author: L. Finesso.}
}
\maketitle

\begin{abstract}
We pose the deterministic, nonparametric, approximation problem for scalar nonnegative input/output systems via finite impulse response convolutions, based on repeated observations of input/output signal pairs. The problem is  converted into a nonnegative matrix factorization with special structure for which we use Csisz\'ar's I-divergence as the criterion of optimality. Conditions are given, on the input/output data, that guarantee the existence and uniqueness of the minimum. We propose a standard algorithm of the alternating minimization type for I-divergence minimization, and study its asymptotic behavior. We also provide a statistical version of the minimization problem and give its large sample properties.
\end{abstract}

%

%
\IEEEpeerreviewmaketitle

%
%
%
%
%

\section{Introduction}\label{section:intro}

Inverse problems are at the core of system modeling and identification. Since the publication of \cite{tikhonov} they have been the subject of a vast technical literature in applied mathematics, engineering, and specialized applied fields. The focus of this paper is on the subclass of  problems  for which the models are linear and time (or space) invariant. Even within this much narrower field the literature is very rich, with many of the contributions leaning towards specific computational aspects of interest for specialized applications.

The goals of the present paper are to pose the problem of approximation of nonnegative i/o system by finite impulse response convolutions, when repeated input/output measurements are available, to propose an algorithm for its solution, and to study its convergence properties. We do not deal with the computational aspects, which must be tailored on the specific application to be of effective value. Our attention will moreover be restricted to nonnegative impulse responses, i.e.\ those for which positive inputs result in positive outputs.

Early contributions for the class of strictly related  nonnegative
deconvolution problems, are \cite{Vardietal1985}, \cite{snyderetal1992} for single input/output observations. Following the choice made in those early contributions the criterion of optimality will be Csisz\'ar's I-divergence, which as argued in~\cite{snyderetal1992} is the best choice for approximation problems under nonnegativity constraints.  From the mathematical point of view the techniques that have been used in~\cite{fs2006} to analyse a nonnegative matrix factorization algorithm are perfectly suited to deal with
the present approximation problem and provide several benefits over the traditional analyses contained in~\cite{snyderetal1992}.

 We provide explicit conditions for the existence and uniqueness of the minimizer of the criterion  in terms of the data. The algorithm that minimizes the informational divergence criterion is of the alternating minimization type, and the optimality conditions (the Pythagorean relations) are satisfied at each step. Exploiting this, we are able to present a proof of convergence which is more transparent than other proofs in the literature, e.g.~\cite{cover1984}, \cite{snyderetal1992}, and \cite{Vardietal1985}. Contrary to previous contributions our treatment allows for $m$  multiple input/\-output pairs. The algorithm for the case $m=1$ has been studied in~\cite{snyderetal1992}. An advantage of allowing multiple input/output pairs is that this setting leads easily to a statistical analysis. In the last section of the paper we provide a statistical version of the minimization problem and give its large sample properties.
%

We emphasize that here we pursue a nonparametric approach to the approximation of a given input/output system by a linear time invariant system. No assumptions on the order, which could as well be infinite, are being made. In doing so we view things from a completely different angle than is usual for the identification or realization of (nonnegative) linear systems, see~\cite{bevenutifarina2004} for instance.
The contributions of the paper are theoretical.  Possible applications of the algorithm are in the field of image processing and emission tomography. For these we refer for instance to~\cite{snyderetal1992,sullivan2007,Vardietal1985}  and the references therein.

A brief summary of the paper follows. In Section~\ref{section: problem} we state the problem and formulate conditions for strict convexity of the objective function, and hence for the existence and uniqueness of the solution. In Section~\ref{sec:lift} the original problem is lifted into a higher dimensional setting, thus making it amenable to alternating minimization. The optimality properties (Pythagoras rules) of the ensuing partial minimization problems are established here. In Section~\ref{section:algorithm} we derive the iterated minimization algorithm combining the solutions of the partial minimizations and we present its first properties. Section~\ref{section:asymptotics} is devoted to the convergence analysis of the algorithm. The  Pythagoras rules facilitate compact and transparent proofs. In the last Section~\ref{section:stats}, taking advantage of the repeated input/output measurements setup, we give a concise treatment of a statistical version of the approximation problem, focusing on its large sample properties.

\section{Problem statement and preliminary results}\label{section: problem}
A discrete time, causal, convolutional system $\mathcal S_h$ maps input sequences  $(u_t)_{t \in \mathbb N} \in  \mathbb R^\mathbb N$ into output sequences  $(y_t)_{t \in \mathbb N} \in  \mathbb R^\mathbb N$, and is completely characterized by an impulse response sequence  $(h_t)_{t \in \mathbb N} \in  \mathbb R^\mathbb N$, such that
\begin{equation} \label{convsys}
y_t =  \mathcal S_h u_t = \sum_{k=0}^t  h_k u_{t-k}, \qquad t \in \mathbb N.
\end{equation}


Rewriting equation~\eqref{convsys}, for $t=0,\ldots,N$, in matrix form, one gets the system of equations
\begin{equation}  \label{eq:matrix}
\begin{pmatrix}
y_0 \\
\vdots \\
\vdots \\
y_N
\end{pmatrix}
=
\begin{pmatrix}
h_0 & 0   & \cdots & \cdots & 0 \\
h_1 & h_0 & 0      & \cdots & 0 \\
\vdots &    \ddots      & \ddots  &  \ddots   & \vdots \\
\vdots &   &  \ddots      & \ddots  &  0   \\
h_N & \cdots & \cdots & h_1 & h_0
\end{pmatrix}
\begin{pmatrix}
u_0 \\
\vdots \\
\vdots \\
u_N
\end{pmatrix},
\end{equation}
compactly written as
\begin{equation}\label{eq:yu}
y=T(h)u,
\end{equation}
having introduced the notations $u=(u_0,\ldots,u_N)^\top$, \,  $y=(y_0,\ldots,y_N)^\top$ and $T(h)$ for the matrix in~\eqref{eq:matrix}. For $m$ input sequences $u^j$, with corresponding output sequences $y^j$, where $j=1,\dots, m$, equation~\eqref{eq:yu} becomes
\begin{equation}\label{eq:YU}
Y=T(h)U,
\end{equation}
where $Y=$$(y^1,\ldots,y^m)$$\in$$\rr^{(N+1)\times m}$ and $U=$$(u^1,\ldots,u^m)$$\in\rr^{(N+1)\times m}$.



\begin{convention} \label{convention}
In expressions containing elements of $U$ the first index is allowed to run out of range, posing $U_{ij}:=0$ for all $i<0$.
\end{convention}

In many practical contexts the inputs and outputs $U$ and $Y$ are directly measured \emph{data}, while $h$ is not known or, more generally, a causal convolutional system $\mathcal S_h$ is not known to exist such that $Y=T(h)U$. In either of these cases an interesting problem is to find $h$ such that the approximate relation
\begin{equation} \label{eq:approx}
Y \approx T(h)U
\end{equation}
is the best possible with respect to a specified loss criterion.

In the paper we concentrate on this problem, under the extra condition that~\eqref{eq:approx} is the approximate representation of the behavior of a positive system, i.e.\ all quantities in~\eqref{eq:approx} are nonnegative real numbers. The goal is the determination of the \emph{best} nonnegative sequence $h=(h_0,\ldots,h_N)^\top$, where the loss criterion, chosen to measure the discrepancy between the left and the right hand side in~\eqref{eq:approx}, is the {\em I-divergence} between nonnegative matrices. See~\cite{cs1991} for a justification from first principles.

For given nonnegative matrices $M$ and $N$ of the same size, $M$ is said to be absolutely continuous with respect to $N$, denoted $M \ll N$, if elementwise~$M_{ij}=0$ for all $(i, j)$ such that $N_{ij}=0$. The {\em I-divergence} between the nonnegative matrices of the same size $M$, and $N$ is defined as
\begin{equation} \label{def:Idiv}
\ii(M||N):= \sum_{ij}\left(M_{ij}\log\frac{M_{ij}}{N_{ij}}-M_{ij}+N_{ij}\right),
\end{equation}
if $M \ll N$, otherwise set $\ii(M||N):=+\infty$
%
In definition~(\ref{def:Idiv}) we also adopt the usual conventions $\frac{0}{0}=0$ and $0\log 0 = 0$.
This leads to
\begin{problem}\label{problem:minh}
For given $Y\geq 0$  and $U \geq 0 $, find a nonnegative vector $h=(h_0,\ldots,h_N)^\top\in \ch := \mathbb R^{N+1}_+$ such that
\[
F(h):=\ii(Y||T(h)U)
\]
is minimized over $\ch$.
\end{problem}

\smallskip
\begin{remark}\label{remark:s}
In Problem \ref{problem:minh} one can assume, without loss of generality, that $S:=\sum_{ij}Y_{ij}=1$. Indeed, for any $S>0$, put $\widetilde{Y}_{ij}=Y_{ij}/S$ and $\widetilde{U}_{ij}=U_{ij}/S$. It then holds that $\ii(Y||T(h)U)=S\ii(\widetilde{Y}||T(h)\widetilde{U})$, and since $S$ does not depend on $h$ the two problems have the same minimizers. This property will be useful in Section~\ref{section:asymptotics}.
\end{remark}

\smallskip
Problem~\ref{problem:minh} is well posed if there exists at least one $h\in\rr^{N+1}_+$ such that $F(h)$ is finite. From definition~(\ref{def:Idiv}) it follows that $F(h)$ is finite if and only if $Y \ll T(h)U$, or equivalently iff $(T(h)U)_{ij}>0$ for all $(i,j)$ such that $Y_{ij}>0$. Since
\begin{equation} \label{eq:thuij}
(T(h)U)_{ij} = \sum_{k=0}^i h_k U_{i-k,j},
\end{equation}
the following condition characterizes the data $(U, Y)$ that produce a well posed Problem~\ref{problem:minh}.
\begin{cond} \label{cond:wp}
For all $(i,j)$ such that $Y_{ij}>0$ there exists $\ell \le i$ such that $U_{\ell j}>0$.
\end{cond}
\noindent
Condition \ref{cond:wp} is rather weak.  In terms of the data sequences it states that if $y_i^j>0$ then $u_\ell^j>0$ for some $\ell \le i$, i.e. if the present output is strictly positive then the present or at least one of the past inputs must be strictly positive. This condition is always satisfied if the data $(U, Y)$ are produced by linear, causal systems.
%

We prove below that, under a  stronger condition on the data $(U, Y)$, the loss $F(h)$ is strictly convex, a property that simplifies the study of the existence and uniqueness of the solution of Problem  \ref{problem:minh}.

\begin{cond} \label{cond:sc}
For all $i\in\{0,\ldots,N\}$ there exists $j\in\{1,\ldots,m\}$ such that $Y_{ij}>0$ and $U_{0j}>0$.
\end{cond}

\noindent
Condition \ref{cond:sc} is strictly stronger than Condition \ref{cond:wp}, but still rather weak. Physically it states that for each time $i$ there exists at least one experiment $j$ with strictly positive initial input $U_{0j}$ and strictly positive output $Y_{ij}$ at time $i$.
This condition holds e.g.\ under the (stronger) assumption that for some experiment $j$, with initial input $U_{0j}>0$, the output trajectory $Y_{ij}$ is strictly positive.

\begin{lemma}\label{lemma:fconvex}
Under Condition \ref{cond:sc} the loss $F(h)$ is strictly convex on its effective domain, that is the set
$\{\,h: \, F(h)< \infty \, \}$.
\end{lemma}

\begin{proof}
The elements $H_{kl}$ of the Hessian $H$ of the loss $F(h)$ are
\[
H_{kl}:=\frac{\partial^2 F}{\partial h_k\partial h_l}(h)
=\sum_{ij}\frac{Y_{ij}}{(T(h)U)_{ij}^2}U_{i-k,j}U_{i-l,j}.
\]
It is enough to show that $H$ is strictly positive definite. Let $x\in\rr^{N+1}$, then
\[
x^\top Hx= \sum_{kl}H_{kl}x_kx_l = \sum_{ij}\frac{Y_{ij}}{(T(h)U)_{ij}^2}(U*x)_{ij}^2,
\]
where $(U*x)_{ij}=\sum_lx_lU_{i-l,j}$. Let $x^\top Hx=0$. By nonnegativity of the summands, this only happens if $\frac{Y_{ij}}{(T(h)U)_{ij}^2}(U*x)_{ij}^2=0$ for all $i,j$. Since $F(h)<\infty$ on its effective domain, we must have $T(h)U_{ij}>0$ as soon as $Y_{ij}>0$. Hence $x^\top Hx=0$ iff $Y_{ij}(U*x)_{ij}=0$ for all $i,j$, which gives a system of linear equations in $x$. For every $i$ fixed and summing over $j$ one explicitly obtains $\sum_k(\sum_j Y_{ij}U_{i-k,j})x_k=0$. This gives a system of equations in which the matrix of coefficients is lower triangular with $\sum_jY_{kj}U_{0j}$ as the $k$-th diagonal element.
Hence this system of equations has $x=0$ as its only solution iff $\sum_jY_{kj}U_{0j}>0$ for all $k$, but the latter constraint is guaranteed by Condition \ref{cond:sc}, hence the Lemma is proved.
\end{proof}

We are now ready to state the existence and uniqueness result. The proof is deferred to Section~\ref{section:algorithm}.
\begin{proposition}\label{proposition:minh}
Assume Condition~\ref{cond:sc} is satisfied, \ then Problem~\ref{problem:minh} admits a unique solution.
\end{proposition}
We write below the standard Kuhn-Tucker necessary conditions for a vector $h$ to be a minimizer of $F(h)$.  Note that, due to the convexity of the divergence $F(\cdot)$ and the concavity of the nonnegativity constraint, the Kuhn-Tucker conditions are sufficient for optimality (see e.g.~\cite[Theorem~2.19]{zangwill}). Condition \ref{cond:sc}, guarantees the strict convexity of $F(\cdot)$ and therefore the uniqueness of the optimizer. Here, and elsewhere in the paper, a dot $\centerdot$ in place of an index denotes summation with respect to the dotted index, e.g. $M_{i\centerdot}:= \sum_j M_{ij}$

\smallskip\noindent
Denoting $\nabla F(h)_k:=\frac{\partial F(h)}{\partial h_k}$, for $k=0,\ldots, N$, the Kuhn Tucker conditions assert that, if the vector $h$ minimizes $F(h)$ subject to the constraints $h_k\ge 0$, then
\begin{align}
\nabla F(h)_k & = 0 \quad \mbox{if} \quad h_k>0, \label{eq:kt1}\\
\nabla F(h)_k & \leq 0 \quad \mbox{if} \quad h_k=0,\label{eq:kt2}
\end{align}
where the partial derivatives $\nabla F(h)_k$ are explicitly given by
\begin{equation}\label{eq:gradient}
\nabla F(h)_k=-\sum_{j=1}^m\sum_{i=k}^{N}\frac{Y_{ij}U_{i-k,j}}{\sum_ph_{p}U_{i-p,j}}+\sum_{l=0}^{N-k}U_{l\centerdot}.
\end{equation}

\begin{example}\label{example:toy}
To illustrate that the minimizers $h$ may be interior points (all $h_k>0$) or boundary points (some $h_k=0$), we consider the following toy example. Let $m=1$ and $N=1$, then $T(h)U$ is a two dimensional vector with components $h_0u_0$ and $h_0u_1+h_1u_0$. The function $F$ is given by
\begin{IEEEeqnarray*}{rCl}
F(h)&=&y_0\log\frac{y_0}{u_0h_0}-y_0+u_0h_0+y_1\log\frac{y_1}{h_0u_1+h_1u_0}  \\
&& - \: y_1+h_0u_1+h_1u_0.
\end{IEEEeqnarray*}
Condition~\ref{cond:wp} for well-posedness reads: if $y_0>0$, then $u_0>0$, and if $y_1>0$ then $u_0>0$ or $u_1>0$. The Condition~\ref{cond:sc} for strict convexity reads: $y_0y_1u_0>0$. One checks by immediate inspection that strict convexity doesn't hold if $y_0=0$ of $y_1=0$.

In this simple case the minimizing $h^*=(h_0^*, h_1^*)^\top$ can be written explicitly by inspection. Since $F(h)\ge 0$, with equality if and only if $Y=T(h)U$, one gets that, if $y_1u_0-y_0u_1\geq 0$, then $h_0^*=\frac{y_0}{u_0}$ and $h_1^*=\frac{y_1u_0-y_0u_1}{u_0^2}$ satisfies the constraints $h^*\ge0$ and attains the minimum $F(h^*)=0$. On the other hand, if $y_1u_0-y_0u_1 < 0$, then the boundary point $h_0^*=\frac{y_0+y_1}{u_0+u_1}$, and $h_1^*=0$ satisfies the constraints $h^*\ge 0$. Checking that $h^*$  satisfies the Kuhn Tucker conditions guarantees that it is a minimizer. From equation~(\ref{eq:gradient}) one gets $\frac{\partial F}{\partial h_0}(h^*)=0$, and
$\frac{\partial F}{\partial h_1}(h^*)=\frac{u_0}{u_1(y_0+y_1)}(u_1y_0-u_0y_1) \geq 0$, in agreement with (\ref{eq:kt1}) and (\ref{eq:kt2}). See also Remark~\ref{remark:m=1} below for more general considerations.
%

\end{example}

\medskip
In solving Problem~\ref{problem:minh}, minimizers $h^*$ at the boundary of $\mathcal H=\mathbb R^{N+1}_+$, i.e. with some zero components, are the rule rather than an exception. This is illustrated in the following remark.

\begin{remark}\label{remark:m=1}
We analyse here the conditions that produce interior and boundary solutions of Problem~\ref{problem:minh}, limiting the discussion to the case $m=1$ which is more transparent. If the minimizer $h$ belongs to the interior of the domain $\mathcal H$, then it can be found imposing that $\nabla F(h)_k=0$ for all $k=0, \dots, N$, i.e. from equation~\eqref{eq:gradient},
\begin{equation}\label{eq:grad1-m=1}
\nabla F(h)_k=-\sum_{i=k}^{N}\frac{y_{i}u_{i-k}}{\sum_ph_{p}u_{i-p}}+\sum_{l=0}^{N-k}u_{l}=0.
\end{equation}
Assume that $u_0>0$. Denoting $t_i:=\sum_p h_{p}u_{i-p}$, the above constraints become
\begin{equation}\label{eq:grad2-m=1}
\nabla F(h)_k=-\sum_{i=k}^{N}\frac{y_{i}u_{i-k}}{t_i}+\sum_{i=k}^{N}u_{i-k} = 0.
\end{equation}
For $k=N$ this reduces to
\begin{equation*}
-\frac{y_{N}u_{0}}{t_N}+u_{0}=0,
\end{equation*}
and one gets $t_N=y_N$. Substitution into equation~(\ref{eq:grad2-m=1}) for $k=N-1$ gives,
\begin{equation*}
-\frac{y_{N-1}u_{0}}{t_{N-1}}+u_{0}-\frac{y_{N}u_{1}}{t_N}+u_{1}=0,
\end{equation*}
and one gets $t_{N-1}=y_{N-1}$. Completing the recursion one gets the system of equations satisfied by the optimal $h$,
\begin{equation}\label{eq:exact}
y_i=t_i = \sum_{p=0}^i h_p u_{i-p}, \qquad \text{for} \,\, i=0, \dots, N.
\end{equation}
In other words the only interior solution, if it exists, corresponds to perfect modeling, $Y=T(h)U$.
 Note that, to find the unknown $h$, system~(\ref{eq:exact}) can be rewritten as follows  \[
\begin{pmatrix}
y_0 \\
\vdots \\
\vdots \\
y_N
\end{pmatrix}
=
\begin{pmatrix}
u_0 & 0   & \cdots & \cdots & 0 \\
u_1 & u_0 & 0      & \cdots & 0 \\
\vdots &    \ddots      & \ddots  &  \ddots   & \vdots \\
\vdots &   &  \ddots      & \ddots  &  0   \\
u_N & \cdots & \cdots & u_1 & u_0
\end{pmatrix}
\begin{pmatrix}
h_0 \\
\vdots \\
\vdots \\
h_N
\end{pmatrix},
\]
which is an alternative way of writing \eqref{eq:matrix}. The computation of the solution by Cramer's rule gives necessary and sufficient conditions on the data $(u, y)$, in terms of a number of determinants, for the existence of a feasible solution, $h \in \mathcal H$. If at least one of these conditions is violated, a feasible solution of Problem~\ref{problem:minh} will necessarily be a boundary point. In this sense boundary point solutions are the rule rather than the exception.
\end{remark}


\section{Lifted version of Problem~\ref{problem:minh}} \label{sec:lift}
\setcounter{equation}{0}

To solve Problem~\ref{problem:minh} we propose an alternating minimization algorithm, following the approach adopted for the derivation of the NMF algorithm in \cite{fs2006} and. In particular, we use a variation on the lifting technique pioneered by \cite{ct1984} and followed in \cite{fs2006}, recasting Problem~\ref{problem:minh} as a double minimization in a larger space. Here and in the following sections bold capitals, e.g.\ $\mathbf M$, will denote matrices (tensors actually) with three indices. The ambient space in which the lifted problem objects live is $\ch_3:=\rr_+^{(N+1)\times (N+1)\times m}$, and specifically on $\bby$, and $\bbw$, two subsets of $\ch_3$ defined below in terms of the given data $(Y, U)$,
$$
\bby  =\left\{\,\,\by\in \ch_3: \quad \by_{i\centerdot j}=Y_{ij}\,\,\right\},
$$
$$
\bbw  =\left\{\bw\in \ch_3: \quad \bw_{ilj} = h_lU_{i-l,j}, \,\,\, \text{for some} \,\, h\in\ch\,\,\right\}.
$$

\bigskip\noindent
\begin{remark}
As a consequence of Convention~\ref{convention}, all $\bw \in \bbw$ have elements $\bw_{ilj}=h_lU_{i-l,j}=0$ for $i<l$..
\end{remark}

\begin{remark}
For any $\bw \in \ch_3$ let $W\in \mathbb R_+^{(N+1)\times m}$ be its marginal, with elements $W_{ij}:=\bw_{i\centerdot j}$. Note that
\begin{equation}\label{eq:W-marginal}
\bw \in \bbw \qquad \Longrightarrow \qquad W_{ij} = \sum_l h_{l}U_{i-l,j}.
\end{equation}
It follows that, if $\mathbf Y \in \bby \cap \bbw$, the data $(Y, U)$ can be described with a perfect model $Y=T(h)U$, since equation~(\ref{eq:W-marginal}) and the definition of $\bby$, imply that $Y_{ij} = \sum_l h_{l}U_{i-l,j}$.
\end{remark}

\bigskip
We consider below two divergence minimization problems in the ambient space $\ch_3$.
\begin{problem}\label{problem:miny}
Given $\bw\in\ch_3$, minimize the divergence $\ii(\by||\bw)$ over $\by\in\bby$.
\end{problem}
\begin{problem}\label{problem:minw}
Given $\by\in\ch_3$, minimize the divergence $\ii(\by||\bw)$ over $\bw\in\bbw$.
\end{problem}

Both problems have explicit solutions. Problem~\ref{problem:miny}, the first, has already been solved in \cite{fs2006}. For ease of reference, we adapt the result below.
\begin{proposition}\label{proposition:sol1}
The solution of Problem~\ref{problem:miny}, denoted $\by^*$ or $\by^*(\bw)$,  satisfies
\[
\by^*_{ilj}=\frac{Y_{ij}}{W_{ij}}\bw_{ilj},
\]
moreover
\begin{equation}\label{eq:ii}
\ii(\by^*(\bw)||\bw) = \ii(Y||W),
\end{equation}
which, if $\bw\in\bbw$, reads
\begin{equation}\label{eq:iih}
\ii(\by^*(\bw)||\bw) = \ii(Y||T(h)U).
\end{equation}
\end{proposition}

The solution of Problem~\ref{problem:minw}, the second, is detailed in the next proposition. Here and elsewhere in the paper we use the notation $$
\alpha_k=\sum_{l=0}^kU_{l\centerdot}, \qquad k=0,\ldots,N.
$$
\begin{proposition}\label{proposition:hstar}
Assume that $U_{0\centerdot}>0$. The solution of Problem~\ref{problem:minw}, denoted $\bw^*$ or $\bw^*(\by)$,  satisfies
\[
\bw^*_{ilj}=h^*_{l}U_{i-l,j}, \quad \text{where} \quad h^*_l=\frac{\by_{\centerdot l\centerdot}}{\alpha_{N-l}},
\]
moreover, if $\by\in\bby$, the vector $h^* \in \mathcal{S}:=\{h\in \ch: \sum_{k=0}^N h_k\alpha_{N-k}=\sum_{ij}Y_{ij}\}$.
\end{proposition}
\begin{proof}
Since $\bw\in\bbw$, we in fact optimize over $h\in\ch$.
Trivial manipulations of the objective function reduce the problem to the explicit minimization of
\[
-\sum_{l=0}^N\by_{\centerdot l\centerdot}\log h_{l}+\sum_{l=0}^Nh_{l}\alpha_{N-l},
\]
which is attained at $h^*$. Finally, if $\by \in \bby$, checking that $h^*\in\mathcal{S}$ is immediate,
$$
\sum_{k=0}^N h^*_k\alpha_{N-k}= \by_{\centerdot \centerdot \centerdot}=\sum_{ij}Y_{ij}.
$$
\end{proof}


\bigskip
Now we can make the connection between the original minimization Problem~\ref{problem:minh} and the two partial minimization Problems~\ref{problem:miny} and~\ref{problem:minw}.
\begin{proposition} \label{MIN=min}
It holds that
\[
\min_{\by\in\bby}\,\min_{\bw\in\bbw}\,\ii(\by||\bw)=\min_{h\in \ch}\, \ii(Y||T(h)U),
\]
moreover, if $h^*$ is the minimizer on the right and $\bw^*$ its correspondent in $\bbw$,
\[
\ii(\by^*(\bw^*)||\bw^*)= \ii(Y||T(h^*)U).
\]
\end{proposition}

\begin{proof}
Fix $\by\in\bby$ and $\bw\in\bbw$, and let $\by^*=\by^*(\bw)$ be the solution of Problem~\ref{problem:miny} with $\bw$ as input. From equation~\eqref{eq:iih}, one has
\begin{align*}
\ii(\by|\bw) & \geq \ii(\by^*(\bw)||\bw) \\
& = \ii(Y||T(h)U) \\
& \geq \inf_{h\in\ch}\ii(Y||T(h)U).
\end{align*}
It follows that
$$
\min_{\by\in\bby}\,\min_{\bw\in\bbw}\,\ii(\by||\bw) \geq \inf_{h\in \ch}\, \ii(Y||T(h)U).
$$
Conversely, fix $h\in \ch$ and let $\bw$ be the corresponding element in $\bbw$, i.e. with $\bw_{ilj}=h_l U_{i-l,j}$ then, again from equation~\eqref{eq:iih},
\begin{align*}
\ii(Y|T(h)U) & = \ii(\by^*(\bw)||\bw) \\
& \geq \min_{\by\in\bby}\,\min_{\bw\in\bbw}\, \ii(\by||\bw),
\end{align*}
which yields
$$
\inf_{h\in \ch}\, \ii(Y||T(h)U)\geq \min_{\by\in\bby}\,\min_{\bw\in\bbw}\, \ii(\by||\bw).
$$
Next we check the value of the minimum. Proposition~\ref{proposition:minh} guarantees the existence of a minimizer of the right hand side, call it $h^*\in\ch$, and let $\bw^*$ be the corresponding element of $\bbw$. Then, using \eqref{eq:iih} once more, one gets $\ii(Y||T(h^*)U)=\ii(\by^*(\bw^*)||\bw^*)$, which shows that $(\by^*(\bw^*),\bw^*)$ is a  minimizing pair.
\end{proof}

\bigskip
The solutions of the two partial minimization problems share the essential {\em Pythagorean property} (see e.g. \cite{csiszarshields} and \cite{c1975}) which, in the present context, is derived below.
\begin{lemma}\label{lemma:pyth}
In Problem~\ref{problem:miny}, with $\bw$ fixed, for all $\by\in\bby$,
\begin{equation}\label{eq:pythy}
\ii(\by||\bw)=\ii(\by||\by^*(\bw))+\ii(\by^*(\bw)||\bw).
\end{equation}
In Problem~\ref{problem:minw}, with $\by$ fixed, for all $\bw\in\bbw$,
\begin{equation}\label{eq:pythw}
\ii(\by||\bw)=\ii(\by||\bw^*(\by))+\ii(\bw^*(\by)||\bw).
\end{equation}
\end{lemma}
\begin{proof}
Equation~\eqref{eq:pythy} follows by a straightforward computation. We proceed to the proof of equation~\eqref{eq:pythw}. We first compute
\begin{IEEEeqnarray}{rCl}
& & \ii(\by||\bw)-\ii(\by||\bw^*(\by)) \nonumber \\
&=& \sum_{ilj}\by_{ilj}\log\frac{\bw^*_{ilj}}{\bw_{ilj}} \nonumber\\
&& + \: \sum_{ilj}\bw_{ilj}-\sum_{ilj}\bw^*_{ilj}\nonumber\\
&=& \sum_{l} \by_{\centerdot l\centerdot}\log \frac{h^*_{l}}{h_{l}}+ \sum_{ilj}\big(\bw_{ilj}-\bw^*_{ilj}\big).\label{eq:iyww}
\end{IEEEeqnarray}
Next we compute
\begin{IEEEeqnarray}{rCl}
& & \ii(\bw^*(\by)||\bw) \nonumber \\
&=& \sum_{ilj}\bw^*_{ilj}\log\frac{\bw^*_{ilj}}{\bw_{ilj}} + \sum_{ilj}\bw_{ilj}-\sum_{ilj}\bw^*_{ilj} \nonumber\\
&=& \sum_{il}\frac{Y_{\centerdot l\centerdot}}{\alpha_{N-l}}U_{i-l,\centerdot}\log\frac{h^*_{l}}{h_{l}} + \sum_{ilj}\big(\bw_{ilj}-\bw^*_{ilj}\big) \nonumber \\
&=& \sum_lY_{\centerdot l\centerdot}\log\frac{h^*_{l}}{h_{l}}+ \sum_{ilj}\big(\bw_{ilj}-\bw^*_{ilj}\big),
\end{IEEEeqnarray}
which coincides with \eqref{eq:iyww}.
\end{proof}

\section{Algorithm}\label{section:algorithm}
\setcounter{equation}{0}

We propose here an iterative algorithm for the solution of the minimization Problem~\ref{problem:minh}.
The algorithm is of the classic alternating minimization type, and is derived using the results of the previous section.
Abstractly, one starts at an initial $\bw^0\in \bbw$, and implements the alternating minimization scheme
$$
\dots \, \bw^t\stackrel{1}{\longrightarrow} \by^t\stackrel{2}{\longrightarrow}  \bw^{t+1}\stackrel{1}{\longrightarrow}  \by^{t+1} \, \dots,
$$
where the superscript $t$ denotes the value at the $t$-th iteration. The arrow $\stackrel{1}{\longrightarrow}$ denotes an instance of the first partial minimization, Problem~\ref{problem:miny}, the matrix at the tail of the arrow is the given input, and the matrix at the head is the optimal solution. For instance $\bw^t\stackrel{1}{\longrightarrow} \by^t$ means that $\by^t=\by^*(\bw^t)$. The meaning of $\stackrel{2}{\longrightarrow}$ is analogous, and represents an instance of the second partial minimization, Problem~\ref{problem:minw}. For instance $\by^t\stackrel{2}{\longrightarrow}  \bw^{t+1}$ means that $\bw^{t+1} = \bw^*(\by^t)$. The hope is that the alternating minimizations produce a sequence of iterates $(\bw^t, \by^t)$ converging to the pair $(\bw^*,\by^*(\bw^*))$ of Proposition~\ref{MIN=min}, thus solving Problem~\ref{problem:minh}. This is indeed the case, as proved in Section~\ref{section:asymptotics}. Here we concentrate on producing a computational version of the algorithm sketched above in abstract terms.

\medskip
Note that, at each iteration, $\bw^t$ is completely specified by the fixed data $U$ and by the vector $h^t \in\ch$. Computationally it is more efficient to work only with the vectors $h^t \in \ch$, one therefore has to shunt the $\by^t$ steps of the alternating minimization, and move directly from $\bw^t$ to $\bw^{t+1}$. This leads to the following scheme.
For given $h^t\in \ch$, define the corresponding $\bw^t_{ilj}=h^t_{l}U_{i-l,j}$ and use it as input in the first partial minimization. The solution, computed according to Proposition~\ref{proposition:sol1}, is
\begin{equation}\label{eq:byt}
\by^t_{ilj}=Y_{ij}\frac{h^t_{l}U_{i-l,j}}{\sum_{p=0}^ih^t_{p}U_{i-p,j}}.
\end{equation}
Use now $\by^t_{ilj}$ as input in the second partial minimization. The solution, computed according to Proposition~\ref{proposition:hstar}, is
\begin{equation}\label{eq:ht}
h^{t+1}_k  =\frac{\by^t_{\centerdot  k\centerdot}}{\sum_{l=0}^{N-k}U_{l\centerdot}},
\end{equation}
with
\begin{equation} \label{eq:htdet}
\by^t_{\centerdot k\centerdot}=\sum_{i= k}^N\sum_{j=1}^m\frac{Y_{ij}U_{i-k,j}}{\sum_ph^t_{p}U_{i-p,j}}h^t_k.
\end{equation}
To shunt the $\by^t$ step it is enough to combine equations \eqref{eq:byt}, \eqref{eq:ht}, and \eqref{eq:htdet} to obtain the following iterative algorithm, solely in terms of $h^t$ vectors and original data $(U,Y)$.
\begin{algorithm}\label{algorithm:h}
$$
h^{t+1}=I(h^t),
$$
where the map $I$ acts on the components of $h^t$ as follows
\begin{equation}\label{eq:rech}
h^{t+1}_k = I_k(h^t):=\frac{h^t_k}{\sum_{l=0}^{N-k}U_{l\centerdot}}\sum_{j=1}^m\sum_{i=k}^{N}\frac{Y_{ij}U_{i-k,j}}{\sum_ph^t_{p}U_{i-p,j}}.
\end{equation}
For further reference it is convenient to introduce the functions $G_k$ defined implicitly as (see equation~\eqref{eq:rech})
\[
I_k(h^t):=h^t_k G_k(h^t).
\]
The initial point $h^0\in \ch$ is chosen such that $F(h^0)<\infty$. If the data satisfy $U_{0\centerdot}>0$, a sufficient condition for $F(h^0)<\infty$ is $h^0>0$ componentwise.
\end{algorithm}

\begin{remark} \label{G-regularity}
Note that, under the assumption $U_{0\centerdot}>0$, the functions $G_k(h)$ are continuous at all points $h$ such that $Y \ll T(h)U$.
\end{remark}

\noindent
\emph{Properties of Algorithm~\ref{algorithm:h}}

\begin{enumerate}

\item
The algorithm decreases the divergence $\ii(Y||T(h^t)U)$ at each step. Indeed, by construction and Propositions~\ref{proposition:sol1} and~\ref{proposition:hstar}, we have
\begin{align}
\ii(Y|T(h^{t+1})U) & =\ii(\by^{t+1}||\bw^{t+1}) \nonumber\\
& \leq\ii(\by^t||\bw^{t+1}) \nonumber\\
& \leq\ii(\by^t||\bw^t)=\ii(Y|T(h^t)U).\label{eq:decreasingdiv}
\end{align}
In Proposition~\ref{proposition:gain} we will exactly quantify the decrease.

\item
If for some $t$ the vector $h^t$ is a perfect model, i.e. $Y=T(h^t)U$, then
\begin{IEEEeqnarray*}{rCl}
\sum_{j=1}^m\sum_{i=k}^N\frac{Y_{ij}U_{i-k,j}}{\sum_p h^t_p U_{i-p,j}}
&=& \sum_{j=1}^m\sum_{i=k}^{N}\frac{\sum_ph^t_{p}U_{i-p,j}U_{i-k,j}}{\sum_ph^t_{p}U_{i-p,j}} \\
&=& \sum_{l=0}^{N-k}U_{l\centerdot},
\end{IEEEeqnarray*}
hence, from equation~\eqref{eq:rech}, $h^{t+1}=h^t$, i.e. perfect models are fixed points of the algorithm.

\item
If for some $t$ the gradient $\nabla F(h^t)=0$, i.e. $h^t$ is a stationary point of $F(h)$,  then, using equation~\eqref{eq:gradient} to rewrite the recursion~\eqref{eq:rech},
\begin{equation}\label{eq:hf}
h^{t+1}_k =h^t_k\left(1-\frac{\nabla F(h^t)_k}{\sum_{l=0}^{N-k}U_{l\centerdot}}\right) = h^t,
\end{equation}
i.e.\ stationary points of $F(h)$ are fixed points of the algorithm. Moreover, we recognize a stability property of the recursion. If $h^t$ is such that $F$ is increasing (decreasing) in the $k$-th coordinate of $h^t$, then $h^{t+1}_k<h^t_k$ ($h^{t+1}_k>h^t_k$).

\item\label{item:simplex}
The vectors $h^t$ belong to the simplex $\mathcal{S}$, as it follows from Proposition~\ref{proposition:hstar}

\item
Assume the condition of Lemma~\ref{lemma:fconvex}. If a starting value $h^0_k>0$, then $h^t_k>0$ for all $t>0$.

\item
We omit the details of the following trivial consistency check. If $N=0$, the solution of Problem~\ref{problem:minh} is $h^*=h^*_0=\frac{Y_{0\centerdot}}{U_{0\centerdot}}$, the algorithm produces $h^1=h^*$, and stays there.

\end{enumerate}

We are now in the position to prove Proposition~\ref{proposition:minh}.
\begin{proof}{\bf of Proposition~\ref{proposition:minh}} We assume that the condition of Lemma~\ref{lemma:fconvex} is satisfied, one can take $h=\one$, i.e.\ $h_k\equiv 1$ to have all elements of $T(\one)U$ positive and hence, for this choice of $h$, the $I$-divergence $F(\one)=\ii(Y||T(\one)U)$ is finite. Take then $h^0=\one$ as a starting value of Algorithm~\ref{algorithm:h}, which at the first step produces $h^1$ with $F(h^1)\leq F(h^0)$ according to Equation~\eqref{eq:decreasingdiv}. Moreover, since $h^1$ is (partly) computed according to the second minimization problem, we have in view of Proposition~\ref{proposition:hstar} that $h^1\in\mathcal{S}$, a compact set. Hence we can confine our search for a minimum of $F$ to $\mathcal{S}$.

The functions $d_{ij}:x\to Y_{ij}\log\frac{Y_{ij}}{x}-Y_{ij}+x$ (for $x\geq 0$) have a minimum at $x=Y_{ij}$, also if $Y_{ij}=0$. Choose a sufficiently small positive $\eps<\min \{Y_{ij}: Y_{ij}>0\}$. Then a minimizer of $F$ has to belong to $\cf=\{h\in\ch: (T(h)U)_{ij}\geq \eps, \mbox{ for all } i,j \mbox{ such that } Y_{ij}>0\}$, and thus finding a minimizer of $F$ can be confined to the  compact set $\mathcal{S}\cap\cf$. We next show that this set is nonempty, for a judiciously chosen $\eps>0$. Let $\lambda>0$ and consider $\lambda\one$. Since $U_{0\centerdot}>0$, we can choose $\lambda$ such that $\lambda\sum_{k=0}^N=S$, hence for this $\lambda$ we have $\lambda\one\in\mathcal{S}$. Redefine, if necessary,  $\eps>0$ such that also $\eps<\min_j (T(\lambda\one)U)_{0j}$, then $\lambda\one\in \cf$, showing that $\mathcal{S}\cap\cf$ is non-void.


Since $F$ is continuous on this set, a minimizer indeed exists. Moreover, since  $F$ is strictly convex, it has a unique minimizer, once there exists one.
\end{proof}
Next we quantify the update gain of Algorithm~\ref{algorithm:h} at each step.
\begin{proposition}\label{proposition:gain}
It holds that
\begin{IEEEeqnarray*}{rCl}
&& \ii(Y||T(h^t)U)-\ii(Y||T(h^{t+1})U) \\
&& = \:\ii(\by^t||\by^{t+1}) + \ii(\bw^{t+1}||\bw^{t}).
\end{IEEEeqnarray*}
\end{proposition}

\begin{proof}
Recall that $\bw^{t+1}$ is the result of the second minimization problem with $\by^t$ given. Invoking Equation~\eqref{eq:pythw}, we have
\begin{equation}\label{eq:pythwa}
\ii(\by^t||\bw^t) = \ii(\by^t||\bw^{t+1})+\ii(\bw^{t+1}||\bw^t).
\end{equation}
On the other hand, $\by^{t+1}$ is the result of the first minimization problem with $\bw^{t+1}$ given. Hence Equation~\eqref{eq:pythy} yields
\begin{equation}\label{eq:pythya}
\ii(\by^t||\bw^{t+1}) = \ii(\by^t||\by^{t+1}) + \ii(\by^{t+1}||\bw^{t+1}).
\end{equation}
Substitution of \eqref{eq:pythya} into \eqref{eq:pythwa} yields
\begin{IEEEeqnarray*}{rCl}
&& \ii(\by^t||\bw^t)  \\
&& =\: \ii(\by^t||\by^{t+1}) + \ii(\by^{t+1}||\bw^{t+1})+\ii(\bw^{t+1}||\bw^t).
\end{IEEEeqnarray*}
To finish the proof apply \eqref{eq:iih} to both $\ii(\by^t||\bw^t)$ and $\ii(\by^{t+1}||\bw^{t+1})$.
\end{proof}
Notice that the update gain is the sum of two non-negative contributions, one from the first minimization and one from the second. The latter term can be given in an alternative expression, which will be useful later (see proof of Lemma~\ref{lemma:fixed}). We have
\begin{align*}
\ii(\bw^{t+1}||\bw^t) & =  \sum_{l=0}^NU_{l\centerdot}\sum_{k=0}^{N-l}(h^{t+1}_{k}\log\frac{h^{t+1}_{k}}{h^t_{k}}-h^{t+1}_{k}+h^{t}_{k}) \\
& = \sum_{k=0}^N(\sum_{l=0}^{N-k}U_{l\centerdot})\ii(h^{t+1}_k||h^t_k)\\
& = \sum_{k=0}^N\alpha_{N-k}\ii(h^{t+1}_k||h^t_k).
\end{align*}
Recall that each $h^t$ belongs to $\mathcal{S}$, since $\sum_{k=0}^N h^t_k\alpha_{N-k}=\sum_{ij}^{N}Y_{ij}=:S$. Let
$$
p^t_k:=\alpha_{N-k}h^t_k/S, \qquad k=0,1\ldots N,
$$
then $p^t:=(p^t_0,\ldots,p^t_N)$ is a probability vector and
\begin{align*}
S\ii(p^{t+1}||p^t) & = S\sum_k p^{t+1}_k\log\frac{p^{t+1}_k}{p^t_k} \\
& = \sum_k \alpha_{N-k}h^{t+1}_k\log\frac{h^{t+1}_k}{h^t_k} \\
& = \sum_k(\alpha_{N-k}\ii(h^{t+1}_k||h^t_k)+p^{t+1}_k-p^t_k)  \\
& = \sum_k\alpha_{N-k}\ii(h^{t+1}_k||h^t_k).
\end{align*}
It follows that
\begin{equation}\label{eq:wp}
\ii(\bw^{t+1}||\bw^t)=S\ii(p^{t+1}||p^t).
\end{equation}

\section{Asymptotics}\label{section:asymptotics}

We turn to the asymptotic behaviour of Algorithm~\ref{algorithm:h}. The main result of the section is Theorem~\ref{thm:limit}.
The preparatory lemmas, much in the spirit of \cite{Vardietal1985}, \cite{snyderetal1992}, and \cite{cover1984}, are typical of this class of problems. See also \cite{sullivan2007} for a recent example. Our proofs, contrary to the cited references, rely heavily on the optimality
results for the partial minimizations (the Pythagoras rules of Lemma~\ref{lemma:pyth}). As a consequence proofs are short and transparent.

First we use the Pythagoras rules for the updates $\by^t$ and $\bw^{t+1}$. Since $\by^t=\by^*(\bw^t)$ and $\bw^{t+1}=\bw^*(\by^t)$, from Lemma~\ref{lemma:pyth} we get the following identities, valid for any $\by\in\bby$ and $\bw\in\bbw$,
\begin{align}
\ii(\by||\bw^t) & = \ii(\by||\by^t)+\ii(\by^t||\bw^t) \label{eq:pythyt}\\
\ii(\by^t||\bw) & = \ii(\by^t||\bw^{t+1})+\ii(\bw^{t+1}||\bw).
\end{align}
Moreover, from Proposition~\ref{proposition:sol1} we also have
\begin{equation}\label{eq:yhtu}
\ii(\by^t||\bw^t)= \ii(Y||T(h^t)U).
\end{equation}
Suppose that $h^\infty$ is a fixed point of  Algorithm~\ref{algorithm:h}, with corresponding $\bw^\infty\in\bbw$ and let $\by^\infty=\by^*(\bw^\infty)$. Then we also have
\begin{equation}
\ii(\by^\infty||\bw^\infty)=\ii(Y||T(h^\infty)U).
\end{equation}
For simplicity throughout this section we assume, without loss of generality, that $S=\sum_{ij}Y_{ij}=1$, see Remark~\ref{remark:s}. Then we have that $p^t_k=\alpha_{N-k}{h^t_k}$. The update equation~\eqref{eq:ht} is equivalent to
\begin{equation}
p^{t+1}_k=\by^t_{\centerdot k\centerdot}.
\end{equation}
In correspondence to the fixed point $h^\infty$, let us define $p^\infty$ as $p^\infty_k=\alpha_{N-k}h^\infty_k$, then
\begin{equation}
p^{\infty}_k=\by^\infty_{\centerdot k\centerdot}.
\end{equation}
Since $p^t$ and $p^\infty$ are probability vectors, by the lumping property of the I-divergence, see~\cite[Lemma~4.1]{csiszarshields}, it holds that
\begin{equation}\label{eq:lump}
\ii(p^\infty||p^{t+1})\leq \ii(\by^\infty||\by^t).
\end{equation}
We will also need the following
\begin{lemma}\label{lemma:fixed}
Limit points of the sequence $(h^t)$ are fixed points of Algorithm~\ref{algorithm:h}.
\end{lemma}
\begin{proof}
Since the divergence $\ii(Y|T(h^t)U)$ is decreasing in $t$, it has a limit. Hence we obtain from Proposition~\ref{proposition:gain} that $\ii(\bw^{t+1}||\bw^t)$$\rightarrow 0$. From \eqref{eq:wp} it follows that  $\ii(p^{t+1}||p^t)\rightarrow 0$. Suppose that $h^\infty$ is a limit point of $(h^t)$, then $p^\infty$ is a limit point of $(p^t)$. Let $\tilde{h}$ be the iteration of the algorithm if $h^t$ is replaced with $h^\infty$ and $\tilde{p}$ be its counterpart, so $\tilde{h}=I(h^\infty)$. By continuity of $I(\cdot)$, which follows from the continuity of the $G_k$, we then get $\ii(\tilde{p}||p^\infty)=0$ and hence $\tilde{p}=p^\infty$, which entails $\tilde{h}=h^\infty$, so $h^\infty$ is a fixed point of the algorithm.
\end{proof}
We are now ready to prove
\begin{lemma}\label{lemma:mono}
Let $h^\infty$ be a limit point of Algorithm~\ref{algorithm:h}, then $\ii(p^\infty||p^t)$ is decreasing in $t$.
\end{lemma}

\begin{proof}
From \eqref{eq:lump} and \eqref{eq:pythyt} with $\by=\by^\infty$ we have
\begin{align*}
\ii(p^\infty||p^{t+1}) & \leq \ii(\by^\infty||\by^t) \\
& = \ii(\by^\infty||\bw^t)-\ii(\by^t||\bw^t).
\end{align*}
Applying the second Pythagorean rule \eqref{eq:pythw} to the first term in the right hand side, with $\by=\by^\infty$ and hence $\bw^*=\bw^\infty$, we get
\[
\ii(Y^\infty||\bw^t) = \ii(\by^\infty||\bw^\infty)+\ii(\bw^\infty||\bw^t).
\]
By Lemma~\ref{lemma:fixed} a limit point of the sequence $(h^t)$ is also a fixed point of the algorithm. Hence we have $\by^\infty=\by^*(\bw^\infty)$ and we deduce from Proposition~\ref{proposition:sol1} that $\ii(\by^\infty||\bw^\infty)=\ii(Y||T(h^\infty)U)$.
A direct computation, similar to that leading to~\eqref{eq:wp}, yields $\ii(\bw^\infty||\bw^t)=\ii(p^\infty||p^t)$. By also using \eqref{eq:yhtu}, we finally obtain
\begin{IEEEeqnarray*}{rCl}
&& \ii(p^\infty||p^{t+1}) \\
&& \phantom{=}  \leq \: \ii(p^\infty||p^t)-\ii(Y||T(h^t)U) +\ii(Y||T(h^\infty)U)\\
&& \phantom{=} \leq \ii(p^\infty||p^t),
\end{IEEEeqnarray*}
since Proposition~\ref{proposition:gain} implies that $\ii(Y||T(h^t)U)$ is decreasing in $t$ and hence $\ii(Y||T(h^\infty)U)\leq\ii(Y||T(h^t)U)$.
\end{proof}
The main result on the asymptotic behavior of Algorithm~\ref{algorithm:h} is given in the next theorem.
\begin{theorem}\label{thm:limit}
The sequence of iterates $h^t$ converges to a limit $h^\infty$ which minimizes $h\to\ii(Y||T(h)U)$.
\end{theorem}

\begin{proof}
Since all $h^t$ belong to the simplex, see property~5 in the list above, which is compact, the sequence $(h^t)$ has a convergent subsequence, $h^{t_n}\to h^\infty$, for some $h^\infty$. For the corresponding sequence  $(p^t)$ sequence it holds that $p^{t_n} \to p^\infty$. By continuity of the I-divergence in the second argument, $\ii(p^\infty||p^{t_n})=\sum_{k:p^\infty_k>0}p^\infty_k\log\frac{p^\infty_k}{p^{t_n}_k}$, we then have $\ii(p^\infty||p^{t_n})\to 0$. The monotonicity result of Lemma~\ref{lemma:mono} then yields $\ii(p^\infty||p^{t})\to 0$, which implies $p^{t}\to p^\infty$, equivalently $h^{t}\to h^\infty$.
Recall from Lemma~\ref{lemma:fixed} that the limit $h^\infty$ is a fixed point of the algorithm. Hence we have from \eqref{eq:hf}
\[
h^{\infty}_k =h^\infty_k\left(1-\frac{\nabla F(h^\infty)_k}{\sum_{l=0}^{N-k}U_{l\centerdot}}\right).
\]
If $h^{\infty}_k>0$, then $\nabla F(h^\infty)_k=0$. We now consider the case where some $h^{\infty}_k=0$. Consider \eqref{eq:rech}, and write it as the product
\[
h^{t+1}_k=h^t_kG_k(h^t).
\]
It follows that $h^{t+1}_k=h^0_k\prod_{j=0}^tG_k(h^j)$. Since we have convergence of the $h^t_k$, we must  have $G_k(h^\infty)\leq 1$, otherwise the product would explode. Indeed, suppose $G_k(h^\infty)>1$, hence $G_k(h^\infty)>1+\eps$ for some $\eps>0$. Continuity of $G_k(\cdot)$ at $h^\infty$, which holds since $F(h^\infty)<\infty$, yields $\lim_{t\to\infty} G_k(h^t)\geq  1+\eps$, hence eventually $G_k(h^t)>1+\eps/2$, which contradicts that the $h^t$ convergence. We conclude $\nabla F(h^\infty)_k\geq 0$. Altogether, we obtain that for the limit $h^\infty$ the Kuhn-Tucker conditions \eqref{eq:kt1}, \eqref{eq:kt2} for $F$ are satisfied. Since these conditions are also sufficient in view of the convexity of $F$, \cite[Theorem~2.9]{zangwill}, $h^\infty$ minimizes $F$.
\end{proof}
Although Theorem~\ref{thm:limit} establishes convergence of the algorithm, it does not give any information on the rate of convergence. In fact, it is possibly a hard grind to get results in this direction. The following example shows that even in a simple case, depending on the exact circumstances, different rates may occur.

\begin{example}
Here we continue the toy Example~\ref{example:toy}.
The update equation \eqref{eq:rech} for $h_1^t$ becomes
\[
h_1^{t+1}= h_1^t\frac{y_1}{h_0^tu_1+h_1^tu_0}.
\]
Assume again the second case, $y_1u_0-y_0u_1 < 0$, and $y_1>0$ to avoid a trivial recursion. Choose $\eps\in (0, \frac{y_0u_1-y_1u_0}{u_0+u_1})$. We know from Theorem~\ref{thm:limit} that $h_0^t\to\frac{y_0+y_1}{u_0+u_1}$ and $h_1^t\to 0$. Hence $h_0^tu_1+h_1^tu_0\to\frac{y_0+y_1}{u_0+u_1}u_1$, and thus for some $t_0>0$ and $t\leq t_0$ one has $h_0^tu_1+h_1^tu_0>\frac{y_0+y_1}{u_0+u_1}u_1-\eps$ and therefore
\[
h_1^{t+1}\leq h_1^t\frac{y_1(u_0+u_1)}{(y_0+y_1)u_1-\eps(u_0+u_1)}=:h^t_kg_\eps.
\]
Hence we have, at least asymptotically, convergence of $h_1^t\to 0$ at an exponential rate, since $g_\eps<1$ by the choice of $\eps$. Note that, in the notation of the proof of Theorem~\ref{thm:limit}, we have $G_1(h^\infty)=\frac{y_1(u_0+u_1)}{(y_0+y_1)u_1}=g_0<1$.

The convergence of the $h_0^t$ could possibly be slower than exponential, since $G_0(h^\infty)=1$. This will be investigated now. The update equation for $h_0^1$ reads
\[
h_0^{t+1}=\frac{y_0}{u_0+u_1}+h_0^t\frac{y_1u_1}{(u_0+u_1)(h_0^tu_1+h_1^tu_0)}.
\]
Let $v_0^t:=h_0^t-h_0^\infty=h_0^t-\frac{y_0+y_1}{u_0+u_1}$. Tedious computations lead to the recursion for $v_0^t$,
\[
v_0^{t+1}=-\frac{y_1u_0}{(u_0+u_1)(h_0^tu_1+h_1^tu_0)}h_1^t.
\]
Since the factor in front of $h_1^t$ stabilizes around its limit value $-\frac{y_1u_0}{u_1(y_0+y_1)}$ and $h_1^t$ converges exponentially fast to zero, the latter property is shared by $v_0^t$.

Next we investigate the case where an exact solution exists, $y_1u_0-y_0u_1 \geq 0$. Let $v_k^t=h_k^t-h_k^\infty$ and $y_1^t=h_0^tu_1+h_1^tu_0$. Putting the $v_k^t$ in a vector $V^t=(v_0^t,v_1^t)^\top$, one arrives after more tedious computations at the recursion
\begin{IEEEeqnarray*}{rCl}
V^{t+1}&=&\frac{u_1}{y_1^t}\begin{pmatrix}\frac{u_0}{u_0+u_1} \\ -1 \end{pmatrix} \begin{pmatrix} h_1^\infty & -h_0^\infty \end{pmatrix}V_t \\
&\approx& \frac{u_1}{y_1}\begin{pmatrix}\frac{u_0}{u_0+u_1} \\ -1 \end{pmatrix} \begin{pmatrix} h_1^\infty & -h_0^\infty \end{pmatrix}V_t=:AV^t.
\end{IEEEeqnarray*}
Clearly the matrix $A$ in front of $V_t$ at the right hand side is singular. Its eigenvalues are $0$ and $\frac{u_1(y_0+y_1)}{(u_0+u_1)y_1}$, where the latter one is smaller than $1$ if we assume the strict inequality $y_1u_0-y_0u_1 > 0$. Hence, also here one has exponential stability.

What is left is the case $y_1u_0-y_0u_1 = 0$. Now the matrix $A$ has an eigenvalue equal to 1. We investigate the exact equation for $V^t$ in this case,
\[
V^{t+1}=\frac{u_1}{y_1^t}\begin{pmatrix} 0 & -\frac{y_0}{u_0+u_1} \\ 0 & \frac{y_0}{u_0} \end{pmatrix} V^t.
\]
It follows that  for $t\geq 1$
\[
v_0^{t} = -\frac{u_0}{u_0+u_1}v_1^{t},
\]
and hence $y_1^t=y_1+\frac{u_0^2}{u_0+u_1}v_1^t$. This leads to the recursion
\[
v_1^{t+1}=\frac{y_1}{y_1+wv_1^t}v_1^t,
\]
with $w=\frac{u_0^2}{u_0+u_1}$. This recursion has the solution
\[
v_1^t=\frac{v_1^0y_1}{wv_1^0t+y_1}.
\]
We conclude that now $v_1^t$ and hence also $v_0^t$ tend to zero at  rate $1/t$ instead of exponentially.
\end{example}

\section{Statistics} \label{section:stats}
\setcounter{equation}{0}

In the previous sections we focussed on the minimization of $\ii(Y||T(h)U)$, where $Y$ and $U$ were given matrices and we presented an algorithm that asymptotically yields the minimizer. In the present section we concentrate on a statistical version of the minimization problem and its large sample properties. Recall that $Y,U\in\rr^{(N+1)\times m}$. We will give limit results for the optimizing $h=h^m$, when $m\to\infty$ and the pair of columns $(Y^i,U^i)$ of $Y,U$ ($i=1,\ldots,m$) form an i.i.d.\ sample. For each fixed $m$, Algorithm~\ref{algorithm:h} can be used to find $h^m$, which now becomes a random vector as well.

Write $\ii(Y||T(h)U)=\sum_{i=1}^m\ii(Y^i||T(h)U^i)$, with the $Y^i$ and $U^i$ the columns of the matrices $Y$ and $U$ respectively. We assume that the pairs $(Y^i,U^i)$ are i.i.d. In what follows, we let, \emph{contrary to the previously employed notation}, $(Y,U)$ be a random vector that has the same distribution as each of the $(Y^i,U^i)$. Moreover we assume for the entries $Y_j$ and $(T(h)U)_j$ of $Y$ and $T(h)U$ the `true' relationship
\begin{equation}\label{eq:delta}
Y_j=(T(h^*)U)_j\delta_j,
\end{equation}
where $h^*$ is an interior point and the $\delta_j\geq 0$ are assumed to be independent of $U$. In the present context it is more appropriate to have a multiplicative disturbance $\delta_j$, than an additive one as in e.g.\ least squares estimation.

The displayed relationship can be summarized as
\[
Y=\Delta T(h^*)U,
\]
where $\Delta$ is diagonal with entries $\delta_j$, and $U$ and $\Delta$ independent. Moreover, we impose $\ee \Delta=I$, the identity matrix, so $\ee\delta_j=1$.

\begin{lemma}
Assume the model \eqref{eq:delta}, $\ee U_j<\infty$, $\ee\delta_j=1$,  and $\ee\delta_j$$\log\delta_j<\infty$. Then it holds that
\begin{IEEEeqnarray*}{rCl}
& & \ee\ii(Y||T(h)U)  =\ee\ii(T(h^*)||T(h)U) \\
&& + \:\sum_j(\ee(T(h^*)U)_j\ee(\delta_j\log\delta_j)-\ee(T(h^*)U)_j).
\end{IEEEeqnarray*}
\end{lemma}

\begin{proof}
Let us first compute $\ee\ii(Y_j||(T(h)U)_j)$. We get
\begin{IEEEeqnarray*}{rCl}
& &\ee\ii(Y_j||(T(h)U)_j) \\
&=& \ee\{Y_j\log\frac{Y_j}{(T(h)U)_j}-Y_j+(T(h)U)_j \} \\
&=& \ee\{(T(h^*)U)_j\delta_j\log\frac{(T(h^*)U)_j\delta_j}{(T(h)U)_j}  \\
 && \!\! -\: (T(h^*)U)_j\delta_j+(T(h)U)_j\} \\
&=& \ee\{(T(h^*)U)_j\delta_j(\log\frac{(T(h^*)U)_j}{(T(h)U)_j}+\log\delta_j) \\
&& \!\! - \: (T(h^*)U)_j\delta_j+(T(h)U)_j\} \\
&=& \ee(T(h^*)U)_j\log\frac{(T(h^*)U)_j}{(T(h)U)_j}\ee\delta_j  \\
&&  + \: \ee(T(h^*)U)_j\ee(\delta_j\log\delta_j)  \\
&& \!\! - \: \ee(T(h^*)U)_j\ee\delta_j+\ee(T(h)U)_j   \\
&=& \big(\ee(T(h^*)U)_j\log\frac{(T(h^*)U)_j}{(T(h)U)_j}-\ee(T(h^*)U)_j\big)\ee\delta_j   \\
&& \!\! + \: \ee(T(h^*)U)_j\ee(\delta_j\log\delta_j)-\ee(T(h^*)U)_j+\ee(T(h)U)_j   \\
&=& \ee\ii((T(h^*)U)_j||(T(h)U)_j)\ee\delta_j-\ee(T(h)U)_j\ee\delta_j   \\
&& \!\!  + \: \ee(T(h^*)U)_j\ee(\delta_j\log\delta_j)-\ee(T(h^*)U)_j+\ee(T(h)U)_j   \\
&=& \ee\ii((T(h^*)U)_j||(T(h)U)_j)  \\
&& \!\! + \: \ee(T(h^*)U)_j\ee(\delta_j\log\delta_j)-\ee(T(h^*)U)_j.
\end{IEEEeqnarray*}
\end{proof}
It follows that minimizing the function $h\mapsto\ee\ii(Y||T(h)U)$ (referred to below as the limit criterion) is equivalent to minimizing $h\mapsto\ee\ii(T(h^*)U||T(h)U)$.

\begin{proposition}\label{prop:limcrit}
Let $\pp(U_0>0)>0$ and $\ee U_j^2<\infty$ for all $j$. The limit criterion $h\mapsto\ee\ii(Y||T(h)U)$ is strictly convex on the set where it is finite (and hence on a neighbourhood of $h^*$) and has a unique minimum for $h=h^*$.
\end{proposition}

\begin{proof}
The proof of strict convexity is similar to the proof of Lemma~\ref{lemma:fconvex}.
We show that the Hessian $H(h)$ at $h$ of the limit criterion is strictly positive definite on the set where the limit criterion is finite. A computation shows that the $kl$-element of this matrix is equal to
\[
H(h)_{kl}=\ee\sum_j
\frac{(T(h^*)U)_j}{(T(h)U)_j^2}U_{j-k}U_{j-l}.
\]
Clearly, $H(h)$ is finite in a neighborhood of $h^*$. Hence
\begin{align*}
x^\top H(h)x
& = \ee\sum_j\frac{(T(h^*)U)_j}{(T(h)U)_j^2}(U*x)_j^2.
\end{align*}
Hence the expression inside the expectation can only be zero if $U*x= 0$ a.s. Using $\pp(U_0>0)>0$, we argue as in the proof of Lemma~\ref{lemma:fconvex} to deduce that $x=0$ iff $x^\top H(h)x=0$.
Clearly, the limit criterion has a minimum equal to zero at $h=h^*$. Conversely, $\ee\ii(T(h^*)U||T(h)U)=0$ iff $\ii(T(h^*)U||T(h)U)=0$ a.s., which happens iff $T(h^*)U=T(h)U$ a.s. Writing this equality elementwise, $(T(h^*)U)_j=(T(h)U)_j=0$, we obtain $h=h^*$ under the condition that $\pp(U_0>0)>0$. We conclude that $h=h^*$ is the unique minimizer if $\pp(U_0>0)>0$.
\end{proof}

\begin{proposition}
Let $\pp(U_0>0)>0$ and $\ee U_j^2<\infty$ for all $j$, moreover assume that $h^*$ is an interior point. The estimators $\hat{h}^m$, defined as the minimizers of the objective function $\sum_{i=1}^m\ii(Y^i||T(h)U^i)$ are consistent. Moreover, this sequence is asymptotically normal, for some positive definite $\Sigma\in\rr^{(N+1)\times (N+1)}$ we have $\sqrt{m}(\hat{h}^m-h^*)\stackrel{d}{\to} N(0,\Sigma)$.
\end{proposition}

\begin{proof}
The limit criterion $h\mapsto\ee\ii(Y||T(h)U)$ is strictly convex, therefore from~\cite[Problem~5.27]{vandervaart} we conclude that the conditions of \cite[Theorem~5.7]{vandervaart} are satisfied and consistency follows.

To show that the estimators $\hat{h}^m$ are asymptotically normal with covariance function as given in \cite[Theorem~5.23]{vandervaart} (although a specific expression for it is not particularly useful), we have to show that the Hessian $H(h^*)$ at $h^*$ of the limit criterion is strictly positive definite. But this follows from the proof of Proposition~\ref{prop:limcrit} upon taking $h=h^*$.
\end{proof}

\end{document}